\documentclass[10.5pt]{amsart}

\newtheorem{theoreme}{Theorem}[section]

\newtheorem{lemme}[theoreme]{Lemma}

\newtheorem{proposition}[theoreme]{Proposition}

\newtheorem{expectation}[theoreme]{Expectation}

\input xy

\xyoption{all}

\usepackage{amscd}
\usepackage{amssymb}
\usepackage{enumitem}
\usepackage{leftidx}

\usepackage[a4paper, left=2cm,right=2cm,bottom=2cm,top=3.5cm]{geometry}
\usepackage{amsthm,array,amssymb,amscd,amsfonts,latexsym, url}
\usepackage{amsmath}
\usepackage{graphicx,epsfig}

\title[A remark on the product by the hyperplane class]{A remark on the product by the hyperplane class in the Chow ring of some complete intersections}
\author{René Mboro}
\date{}
\begin{document}

\begin{abstract} By classical calculation, for a smooth hypersurface $Y\subset \mathbb P^{n+1}_{\mathbb C}$, the product by the hyperplane class is zero on homologically trivial rational cycles i.e. $H_{|Y}\cdot :{\rm CH}_i(Y)_{hom,\mathbb Q}\rightarrow {\rm CH}_{i-1}(Y)_{hom,\mathbb Q}$ is $0$ for any $i$. This note extends that result to some complete intersections.
\end{abstract}
\maketitle
\section{Introduction}
In this note, we study the multiplication by the hyperplane class in the Chow ring of some complete intersections. Namely, we are interested in the following question: for any smooth hypersurface $\iota: Y\hookrightarrow \mathbb P^{n+1}_{\mathbb C}$ of degree $d$, denoting $H=c_1(\mathcal O_{\mathbb P^{n+1}}(1))$, according to the self-intersection formula (\cite[Corollary 6.3]{fulton}) for Chow groups, the product $dH_{|Y}\cdot:{\rm CH}_i(Y)\rightarrow {\rm CH}_{i-1}(Y)$ by a multiple of the hyperplane class, is equal to the composition $\iota^*\iota*$ so that, since the cycle class gives an isomorphism ${\rm CH}^*(\mathbb P^{n+1})\overset{\sim}{\rightarrow} H^{2*}(\mathbb P^{n+1},\mathbb Z)$, the product by the hyperplane class $H_{|Y}\cdot:{\rm CH}_i(Y)_{hom,\mathbb Q}\rightarrow {\rm CH}_{i-1}(Y)_{hom,\mathbb Q}$, where ${\rm CH}_i(Y)_{hom,\mathbb Q}:={\rm ker}([\ ]:{\rm CH}_i(Y)\otimes \mathbb Q\rightarrow H^{2(n-i)}(Y,\mathbb Q))$, is zero.\\
\indent Likewise, when $\iota:Y\hookrightarrow \mathbb P^{n+r}_{\mathbb C}$ is a smooth complete intersection of dimension $n$, ${\rm deg}(Y)H_{|Y}^r:{\rm CH}_i(X)_{hom,\mathbb Q}\rightarrow {\rm CH}_{i-r}(Y)_{hom,\mathbb Q}$ is trivial. The following question (due to Voisin) is natural: what is the index of nilpotency of $H_{|Y}\cdot$ in $\oplus_i {\rm CH}^i(Y)_{hom,\mathbb Q}$?\\
\indent Again by the self-intersection formula, denoting $j:X\hookrightarrow Y$ the inclusion of a smooth hypersurface, (up to a rational number) $H_{|Y}\cdot=j_*j^*$ on ${\rm CH}_i(Y)_{hom,\mathbb Q}$.\\
\indent Now, the Bloch-Beilinson conjecture (see for example \cite[Conjecture 23.22]{vois_lect}) predicts the existence of a natural decreasing filtration $F^\bullet{\rm CH}^k(Z)_{\mathbb Q}$ on the Chow group ${\rm CH}^k(Z)_{\mathbb Q}:={\rm CH}^k(Z)\otimes \mathbb Q$ of any smooth projective variety $Z$ ($\forall k$) such that $F^1{\rm CH}^k(Z)_{\mathbb Q}={\rm CH}^k(Z)_{hom,\mathbb Q}$ and
\begin{enumerate}
\item $F^{k+1}{\rm CH}^k(Z)_{\mathbb Q}=0$;
\item for any other smooth projective variety $W$ and $\Gamma\in {\rm CH}^{\ell}(Z\times W)$, $\Gamma_*(F^i{\rm CH}^k(Z)_{\mathbb Q})\subset F^i{\rm CH}^{k+\ell-{\rm dim}(Z)}(W)_{\mathbb Q}$;
\item moreover, the induced map ${\rm Gr}^i_F(\Gamma_*):{\rm Gr}^i_F({\rm CH}^k(Z)_{\mathbb Q})\rightarrow {\rm Gr}^i_F({\rm CH}^{k+\ell-{\rm dim}(Z)}(W)_{\mathbb Q})$ is $0$ as soon as $[\Gamma]_*$ vanishes on $H^{r,s}(Z)$, $\forall s\leq k-i$ and $r+s=2k-i$.
\end{enumerate}
\textit{}\\ 
\indent According to the conjecture, to understand the action of $\Gamma$ on $F^1{\rm CH}^k(Z)_{\mathbb Q}={\rm CH}^k(Z)_{hom,\mathbb Q}$, it is sufficient to analyse its action on $H^{r,s}(Z)$ for all $i\geq 1$, $s\leq k-i$ and $r+s=2k-i$.\\
\indent With these restrictions, $(r,s)$ cannot be of the form $(p,p)$ as, writing $i=2m$, we would get $r+s=2(k-m)$ and $s\leq k-i<k-m$.\\
\indent As the non-trivial cohomology of complete intersections is concentrated in their middle cohomology and for the inclusion of an hypersurface $j:X\hookrightarrow Y$, as above, $[\Gamma_j]_*$ sends $H^{dim(X)}(X,\mathbb Q)$ to a trivial Hodge structure, the Bloch-Beilinson conjecture therefore predicts, in particular:
\begin{expectation}\label{conj_intro} Let $Y\subset \mathbb P^{n+r}_{\mathbb C}$ be a smooth complete intersection then $H_{|Y}\cdot:{\rm CH}_i(Y)_{hom,\mathbb Q}\rightarrow {\rm CH}_{i-1}(Y)_{hom,\mathbb Q}$ is zero.
\end{expectation}
The expectation is true for complete intersection surfaces since for such a surface $S$ we have $CH^1(S)_{hom}=0$. We prove:
\begin{theoreme}\label{thm_1} Let $Y\subset \mathbb P^{n+r}$ be a general (smooth) complete intersection of dimension $n\geq 3$ and type $(d_1,\dots,d_r)$, with $d_1\leq \cdots\leq d_r$, $2\leq r< n$ and $d_r>2$. Then for $H_{|Y}\cdot:{\rm CH}_i(Y)_{\mathbb Q}\rightarrow {\rm CH}_{i-1}(Y)_{\mathbb Q}$, ${\rm Im}(H_{|Y}\cdot)\subset \mathbb Q\cdot H_{|Y}^{n-(i-1)}$.\\
\indent In particular, $H_{|Y}\cdot:{\rm CH}_i(Y)_{hom, \mathbb Q}\rightarrow {\rm CH}_{i-1}(Y)_{hom, \mathbb Q}$ is zero.
\end{theoreme}
The proof of the theorem goes by decomposing the action of the graph $\Gamma\in {\rm CH}_{n-1}(X\times Y)$ of the inclusion of a hypersurface $X\subset Y$ of high degree, using the relation of $X$ and $Y$ with the lines of $\mathbb P^{n+r}$.\\
\indent However, to have a control on the coefficient of $\Gamma$ in the relation we want to exhibit with the lines, we examine rather the correspondence $\gamma:=\{(x,x,x)\in X\times X\times Y, x\in X\}$, which satisfies $\Gamma=pr_{23,*}(\gamma)$. As a $X$ of degree high enough has, by Mumford theorem (rather its generalization by Roitman \cite{roitman}, see also \cite[Th\'eor\`eme 22.16 and Corollaire 22.18]{vois_lect}), non trivial $0$-cycles $z\in {\rm CH}_0(X)$ and for such a cycle $\gamma_*(z\times Y)=z\in {\rm CH}_0(X)$, we can have access to the coefficient of $\gamma$ in a relation between correspondences in $X\times X\times Y$.\\
\indent The method is inspired by \cite{vois_decomp} and \cite{fu_decomp}, where the respective authors were interested in exhibiting a relation between the small diagonal of a hypersurface (resp. complete intersections) $Y$ and the cycle made up of points of $Y$ lying on a line contained in $Y$. Here, in order to decompose $\gamma$, we use rather, the lines that have some special contact data with the smallest complete intersection $X$.\\ 
\indent The method seems quite artificial (probably not the most appropriate) but the result it yields is a strong evidence for Expectation \ref{conj_intro}.\\

In the rest of the note, all the Chow groups will be considered with rational coefficients.

\section{Decompostion of the graph of the inclusion of a hypersurface}
Let $Y\subset \mathbb P^{n+r}$ be a smooth complete intersection of dimension $n\geq 3$ and type $(d_1,\dots,d_r)$, with $2\leq d_1\leq \cdots\leq d_r$, $2\leq r< n$ and $d_r>2$. Set $d={\rm max}\{d_r,2(n+r)\}$ and consider a general smooth hypersurface $X\in |\mathcal O_Y(d)|$ that, in particular contains no line. It is possible since a line lying on $X$ would also lie on a general degree $d$ hypersurface of $\mathbb P^{n+r}$. But the variety of lines of the latter is empty since it is cut out from the Grassmannian $G(2,n+r+1)$ by a regular section of ${\rm Sym}^d\mathcal E_2$, which is a rank $d+1>dim(G(2,n+r+1))$ vector bundle, where $\mathcal E_2$ is the rank $2$ quotient bundle on $G(2,n+r+1)$.\\
\indent Let us define $\gamma=\{(x,x,x)\in X\times X\times Y,\ x\in X\}$ and $\Gamma=\{(x,x)\in X\times Y,\ x\in X\}$. As explained in the introduction, we begin by exhibiting a relation in ${\rm CH}_{n-1}(X\times X\times Y)_{\mathbb Q}$ between $\gamma$ and a cycle constructed using the lines of $\mathbb P^{n+r}$.\\

\indent We start by collecting some constructions and results from \cite{fu_decomp}. Let us define $$W:=\{(x_1,x_2,x_3)\in (\mathbb P^{n+r})^3\backslash \delta_{\mathbb P^{n+r}},\ x_1,x_2,x_3\ {\rm are\ colinear}\}$$ where $\delta_{\mathbb P^{n+r}}=\{(y,y,y)\in (\mathbb P^{n+r})^3, y\in \mathbb P^{n+r}\}$ is the small diagonal of $\mathbb P^{n+r}$ and let us denote $W_0\subset W$ the open subset where the $x_i$'s are pairwise distinct. Then $W$ is readily seen to be the variety $\mathbb P(\mathcal E_2)\times_{G(2,n+r+1)}\mathbb P(\mathcal E_2)\times_{G(2,n+r+1)}\mathbb P(\mathcal E_2)\backslash \delta_{\mathbb P(\mathcal E_2)}$, where again $\delta_{\mathbb P(\mathcal E_2)}=\{([\ell],x),([\ell],x),([\ell],x))\in \mathbb P(\mathcal E_2)\times_{G(2,n+r+1)}\mathbb P(\mathcal E_2)\times_{G(2,n+r+1)}\mathbb P(\mathcal E_2),\ ([\ell],x)\in \mathbb P(\mathcal E_2)\}$ is the small diagonal. In particular, $W$ is smooth irreducible of dimension $3+dim(G(2,n+r+1))=2n+2r+1$.\\
\indent We have a natural morphism $t:W\rightarrow G(2,n+r+1)$ and we can consider the diagram: 
\begin{equation}\label{diag_W_G_P}
\xymatrix{\mathbb P(t^*\mathcal E_2)\ar[d]^{p_W}\ar[r]^{q_W} &\mathbb P^{n+r}\\ W &}
\end{equation}
The $\mathbb P^1$-bundle $p_W$ admits three taulogical sections $\sigma_{W,i}:W\rightarrow \mathbb P(t^*\mathcal E_2)$ determined by the point $x_i$, for $i=1,2,3$. Let us denote $D_{W,i}=\sigma_{W,i}(W)$, $i=1,2,3$; for each $i=1,2,3$, $D_{W,i}$ is a divisor in $\mathbb P(t^*\mathcal E_2)$ whose class in ${\rm Pic}(\mathbb P(t^*\mathcal E_2))$ is of the form$q_W^*\mathcal O_{\mathbb P^{n+r}}(1)\otimes p_{W}^*\alpha_i$ for some line bundle $\alpha_i$ on $W$.\\
\indent Let us also introduce $V_0:=\{(x_1,x_2,y)\in X^2\times Y\backslash \gamma,\ x_1,x_2,y\ {\rm are\ colinear\ and\ pairwise\ distinct}\}$ its closure $V:=\overline{V_0}$ in $X^2\times Y\backslash \gamma$. We have $V\backslash V_0=A_{12}\cup B_{13}\cup B_{23}$, with $A_{12}=\{(x,x,y)\in X^2\times Y\backslash \gamma,\ \ell_{x,y}\ {\rm is\ tangent\ to}\ X\ {\rm at}\ x\}$, $B_{13}=\{(x_1,x_2,x_1)\in X^2\times Y\backslash\gamma,\ \ell_{x_1,x_2}\ {\rm is\ tangent\ to}\ Y\ {\rm at}\ x_1\}$ and $B_{23}=\{(x_1,x_2,x_2)\in X^2\times Y\backslash\gamma,\ \ell_{x_1,x_2}\ {\rm is\ tangent\ to}\ Y\ {\rm at}\ x_2\}$, where for any pair $(a,b)\in (\mathbb P^{n+r})^2\backslash \Delta_{\mathbb P^{n+r}}$, $\ell_{a,b}$ denotes the line defined by the pair. We have the following:

\begin{lemme}\label{lem_V_0_non_empty} For any smooth complete intersection $Y$ of type $(d_1,\dots,d_r)$ and any hypersurface $X\in |\mathcal O_Y(d)|$, $V_0$ is non-empty and of dimension at least $2n-r-1$. For $Y$ and $X$ general, $V_0$ is smooth, irreducible of dimension $2n-r-1$.
\end{lemme}
\begin{proof} Denoting $s_{W,j}\in |\mathcal O_{\mathbb P(t^*\mathcal E_2)}(D_{W,j})|$ the section defining $D_{W,j}$, $j=1,2,3$, we have, for each $i=1,\dots r$ an inclusion of lines bundles $q_W^*\mathcal O_{\mathbb P^{n+r}}(d_i)\otimes \mathcal O_{\mathbb P(t^*\mathcal E_2)}(-D_{W,1}-D_{W,2}-D_{W,3})\hookrightarrow q_W^*\mathcal O_{\mathbb P^{n+r}}(d_i)$ given by $s_{W,1}\otimes s_{W,2}\otimes s_{W,3}$ and an inclusion $q_W^*\mathcal O_{\mathbb P^{n+r}}(d)\otimes \mathcal O_{\mathbb P(t^*\mathcal E_2)}(-D_{W,1}-D_{W,2})\overset{s_{W,1}\otimes s_{W,2}}{\hookrightarrow} q_W^*\mathcal O_{\mathbb P^{n+r}}(d)$. Let us introduce the following vector bundles on $W$ $$\begin{aligned}E''&:=p_{W,*}(\oplus_i q_W^*\mathcal O_{\mathbb P^{n+r}}(d_i)\otimes \mathcal O_{\mathbb P(t^*\mathcal E_2)}(-D_{W,1}-D_{W,2}-D_{W,3})\oplus q_W^*\mathcal O_{\mathbb P^{n+r}}(d)\otimes \mathcal O_{\mathbb P(t^*\mathcal E_2)}(-D_{W,1}-D_{W,2}))\\
&\simeq \oplus_i {\rm Sym}^{d_i-3}\mathcal E_2\otimes \alpha_1\otimes\alpha_2\otimes\alpha_3\oplus {\rm Sym}^{d-2}\mathcal E_2\otimes \alpha_1\otimes \alpha_2\end{aligned}$$ (with ${\rm Sym}^{-1}\mathcal E_2=0$) and
$$\begin{aligned}E':=p_{W,*}(\oplus_i q_W^*\mathcal O_{\mathbb P^{n+r}}(d_i)\oplus q_W^*\mathcal O_{\mathbb P^{n+r}}(d))\simeq \oplus_i {\rm Sym}^{d_i}\mathcal E_2\oplus {\rm Sym}^d\mathcal E_2.\end{aligned}$$
The above inclusions of lines bundles give rise to an inclusion of vector bundles $E''\hookrightarrow E'$. The quotient $E\simeq E'/E''$ is a globally generated ($E'$ is) vector bundle of rank $\sum_{i=1}^r(d_i+1)-(d_i-3+1)+(d+1)-(d-2+1)=3r+2$. ``The'' equations of $Y$ and $X$ give rise to a section $\eta_{X,Y}\in H^0(W,E')$ and its projection to a section $\overline{\eta_{X,Y}}\in H^0(W,E)$. The zero locus of $\overline{\eta_{X,Y|W_0}}$ is supported on $V_0$ and has codimension at most $3r+2$. In particular $V_0$ is of dimension at least $2n-r-1>0$.\\
\indent Since $E$ is globally generated, by Bertini type theorems, for $Y$ and $X$ general, the zero locus of $\overline{\eta_{X,Y|W_0}}$ is smooth of the expected dimension $2n-r-1$. Moreover, it is easy to see that any two points of $V_0$ can be connected by chain of curves in particular $V_0$ is connected.
\end{proof}

The following is an analog of \cite[Lemma 1.3]{fu_decomp}
\begin{proposition}\label{prop_intersect_V_diagonals_excess_bundles} (1) The intersection scheme of $W$ and $X^2\times Y\backslash \gamma$ in $(\mathbb P^{n+r})^3\backslash \delta_{\mathbb P^{n+r}}$ has four irreducible components: $$W\cap (X^2\times Y\backslash \gamma)= V\cup \Delta_{12}\cup \Gamma_{13}\cup \Gamma_{23}$$
where $\Delta_{12}:=\{(x,x,y)\in X^2\times Y\backslash \gamma,\ (x,y)\in X\times Y\backslash \Gamma\}$, $\Gamma_{13}:=\{(x_1,x_2,x_1)\in X^2\times Y\backslash \gamma,\ (x_1,x_2)\in X^2\backslash \Delta_X\}$ and $\Gamma_{23}:=\{(x_1,x_2,x_2)\in X^2\times Y\backslash \gamma,\ (x_1,x_2)\in X^2\backslash \Delta_X\}$.\\
\indent For $Y$ and $X$ general, the intersection along $V$ is transversal. The intersection along $\Delta_{12}$ has excess dimension $r$ and multiplicity $1$ and the intersection along the $\Gamma_{i3}$'s ($i=1,2$) has excess dimension $r-1$ and their respective multiplicity is $1$.\\

\indent (2) When $\Delta_{12}$ is identified with $X\times Y\backslash \Gamma$, the excess normal sheaf, in the sense of \cite{fulton}, along $\Delta_{12}\backslash V$, that we will denote $Exc(\Delta_{12}\backslash V)$ is a rank $r$ vector bundle isomorphic to the quotient $pr_X^*N_{X/\mathbb P^{n+r}}/(pr_X^*\mathcal O_X(1)\otimes pr_Y^*\mathcal O_Y(-1))$.\\

\indent (3) Under the identification $\Gamma_{i3}\simeq X\times X\backslash \Delta_X$, the excess normal sheaf along $\Gamma_{i3}\backslash V$ is a rank $r-1$ vector bundle isomorphic to the quotient $Exc(\Gamma_{i3}\backslash V)\simeq pr_1^*N_{Y/\mathbb P^{n+r}|X}/(pr_1^*\mathcal O_X(1)\otimes pr_2^*\mathcal O_X(-1))$.
\end{proposition}
\begin{proof} (1) It follows essentially from Lemma \ref{lem_V_0_non_empty} (see \cite[Lemma 1.3]{fu_decomp}).\\

\indent (2) We follow the arguments of the proof of \cite[Lemma 1.6]{fu_decomp}. Let us consider the commutative diagram:
$$\xymatrix{X\times Y\backslash \Gamma\ar@{^{(}->}[dr]^{i_2'}\ar@{^{(}->}[rr]^{i_2}\ar@{^{(}->}[dd]^{j} & &W\ar@{^{(}->}[dd]^i\\
 &(\mathbb P^{n+r})^2\backslash \Delta_{\mathbb P^{n+r}}\ar@{^{(}->}[ur]^{i_2''} &\\
 X^2\times Y\backslash\gamma\ar@{^{(}->}[rr]_{i_1} &  &(\mathbb P^{n+r})^3\backslash \delta_{\mathbb P^{n+r}}}$$
As $i_1$ and $i_2$ are regular imbeddings, according to \cite[6.3]{fulton}, $Exc(\Delta_{12}\backslash V)\simeq j^*N_{i_1}/N_{i_2}$ where we denote $N_f$ the normal bundle associated to the regular imbedding $f$.\\
\indent Now $j^*N_{i_1}\simeq j^*(pr_1^{X^2\times Y,*}N_{X/\mathbb P^{n+r}}\oplus pr_2^{X^2\times Y,*}N_{X/\mathbb P^{n+r}}\oplus pr_3^{X^2\times Y,*}N_{Y/\mathbb P^{n+r}})\simeq (pr_X^*N_{X/\mathbb P^{n+r}})^{\oplus 2}\oplus pr_Y^*N_{Y/\mathbb P^{n+r}}$.\\
The normal bundle $N_{i_2}$ sits in the exact sequence $$0\rightarrow N_{i_2'}\rightarrow N_{i_2}\rightarrow N_{i_2''|X\times Y\backslash \Gamma}\rightarrow 0.$$
We have $N_{i_2'}\simeq pr_X^*N_{X/\mathbb P^{n+r}}\oplus pr_Y^*N_{Y/\mathbb P^{n+r}}$ and for $i_2''$, using the isomorphisms $$W\simeq \mathbb P(\mathcal E_2)\times_{G(2,n+r+1)}\mathbb P(\mathcal E_2)\times_{G(2,n+r+1)}\mathbb P(\mathcal E_2)\backslash \delta_{\mathbb P(\mathcal E_2)}$$ and $$(\mathbb P^{n+r})^2\backslash \Delta_{\mathbb P^{n+r}}\overset{e}{\xrightarrow{\sim}}\mathbb P(\mathcal E_2)\times_{G(2,n+r+1)}\mathbb P(\mathcal E_2)\backslash \Delta_{\mathbb P(\mathcal E_2)}$$
(that associates to a pair $x_1\neq x_2$ the lines they determine) we see that $N_{i_2''}\simeq e^*pr^{\mathbb P(\mathcal E_2)\times_G \mathbb P(\mathcal E_2),*}_1T_{\mathbb P(\mathcal E_2)/G(2,n+r+1)}$.\\
\indent The evaluation of a linear polynomial in two distinct points of a line gives $pr_1^{\mathbb P(\mathcal E_2)\times_G \mathbb P(\mathcal E_2),*}p_G^*\mathcal E_2\simeq pr_1^{\mathbb P(\mathcal E_2)\times_G \mathbb P(\mathcal E_2),*}q_G^*\mathcal O_{\mathbb P^{n+r}}(1)\oplus pr_2^{\mathbb P(\mathcal E_2)\times_G \mathbb P(\mathcal E_2),*}q_G^*\mathcal O_{\mathbb P^{n+r}}(1)$ with reference to $\xymatrix{\mathbb P(\mathcal E_2)\ar[d]^{p_G}\ar[r]^{q_G} &\mathbb P^{n+r}\\G(2,n+r+1) &}$. Hence, using the usual exact sequence for the relative tangent bundle $$0\rightarrow O_{\mathbb P(\mathcal E_2)}\rightarrow p_G^*\mathcal E_2^\vee\otimes q_G^*\mathcal O_{\mathbb P^{n+r}}(1)\rightarrow T_{\mathbb P(\mathcal E_2)/G(2,n+r+1)}\rightarrow 0$$ we get $e^*pr_1^{\mathbb P(\mathcal E_2)\times_G \mathbb P(\mathcal E_2),*}T_{\mathbb P(\mathcal E_2)/G(2,n+r+1)}\simeq pr_2^*\mathcal O_{\mathbb P^{n+r}}(-1)\otimes pr_1^*\mathcal O_{\mathbb P^{n+r}}(1)$. Putting everything together, we get $Exc(\Delta_{12}\backslash V)=pr_X^*N_{X/\mathbb P^{n+r}}/(pr_Y^*\mathcal O_Y(-1)\otimes pr_X^*\mathcal O_X(1))$.\\

\indent (3) In view of the isomorphism $\Gamma_{13}\xrightarrow{\sim} X\times X\backslash \Delta_X$, $(x_1,x_2,x_1)\mapsto (x_1,x_2)$, let us consider the diagram:
  $$\xymatrix{X\times X\backslash \Delta_X\ar@{^{(}->}[dr]^{i_2'}\ar@{^{(}->}[rr]^{i_2}\ar@{^{(}->}[dd]^{j} & &W\ar@{^{(}->}[dd]^i\\
 &(\mathbb P^{n+r})^2\backslash \Delta_{\mathbb P^{n+r}}\ar@{^{(}->}[ur]^{i_2''} &\\
 X^2\times Y\backslash\gamma\ar@{^{(}->}[rr]_{i_1} &  &(\mathbb P^{n+r})^3\backslash \delta_{\mathbb P^{n+r}}}$$
 As $i_1$ and $i_2$ are regular imbeddings, according to \cite[6.3]{fulton}, $Exc(\Gamma_{13}\backslash V)=j^*N_{i_1}/N_{i_2}$. We have $$\begin{aligned}j^*N_{i_1}&\simeq j^*(pr_1^{X^2\times Y,*}N_{X/\mathbb P^{n+r}}\oplus pr_2^{X^2\times Y,*}N_{X/\mathbb P^{n+r}}\oplus pr_3^{X^2\times Y,*}N_{Y/\mathbb P^{n+r}})\\
 &\simeq pr_1^{X^2,*}(N_{X/\mathbb P^{n+r}}\oplus N_{Y/\mathbb P^{n+r}|X})\oplus pr_2^{X^2,*}N_{X/\mathbb P^{n+r}}\end{aligned}$$
 The normal bundle $N_{i_2}$ sits again in the exact sequence $$0\rightarrow N_{i_2'}\rightarrow N_{i_2}\rightarrow N_{i_2''|X^2\backslash \Delta_X}\rightarrow 0.$$ We have $N_{i_2'}\simeq pr_1^*N_{X/\mathbb P^{n+r}}\oplus pr_2^*N_{X/\mathbb P^{n+r}}$ and $N_{i_2''}$ has already been computed. Hence $Exc(\Gamma_{13}\backslash V)=pr_1^{X^2,*}N_{Y/\mathbb P^{n+r}|X}/(pr_1^{X^2,*}\mathcal O_X(1)\otimes pr_2^{X^2,*}\mathcal O_X(-1))$.\\
\indent Likewise $Exc(\Gamma_{23}\backslash V)=pr_2^*N_{Y/\mathbb P^{n+r}|X}/(pr_1^{X^2,*}\mathcal O_X(1)\otimes pr_2^{X^2,*}\mathcal O_X(-1))$.
\end{proof}

Let us introduce \begin{equation}\label{def_Z_0}Z_0:=\{(x_1,x_2,y)\in V_0,\ P_{d|\ell_{x_1,x_2,y}}\in L_{x_1}^{n-r+1}L_{x_2}H^0(\mathcal O_{\ell_{x_1,x_2,y}}(d-n+r-2))\}\end{equation} where $P_d\in |\mathcal O_{\mathbb P^{n+r}}(d)|$ is a section whose restriction to $Y$ defines $X$, $\ell_{x_1,x_2,y}$ is the line determined by $(x_1,x_2,y)$ and $L_{x_i}$ is the (linear) equation of $x_i\in \ell_{x_1,x_2,y}$. Set $Z=\overline{Z_0}$ in $V$.\\
\indent Let us also introduce the following vector bundles on $W$ (with notations from (\ref{diag_W_G_P}))
$$F':=p_{W,*}(q_W^*\mathcal O_{\mathbb P^{n+r}}(d)\otimes \mathcal O_{\mathbb P(t^*\mathcal E_2)}(-D_{W,1}-D_{W,2}))$$ and $$F'':=p_{W,*}(q_W^*\mathcal O_{\mathbb P^{n+r}}(d)\otimes \mathcal O_{\mathbb P(t^*\mathcal E_2)}(-(n-r+1)D_{W,1}-D_{W,2})).$$
The inclusions of lines bundles $q_W^*\mathcal O_{\mathbb P^{n+r}}(d)\otimes \mathcal O_{\mathbb P(t^*\mathcal E_2)}(-(n-r+1)D_{W,1}-D_{W,2})\overset{s_{W,1}^{\otimes (n-r)}}{\hookrightarrow} q_W^*\mathcal O_{\mathbb P^{n+r}}(d)\otimes \mathcal O_{\mathbb P(t^*\mathcal E_2)}(-D_{W,1}-D_{W,2})$ give rise to an inclusion $F''\hookrightarrow F'$. Let us denote $F:=F'/F''$ the quotient sheaf. It is a rank $(d-2+1)-(d-n+r-2+1)=n-r$ vector bundle on $W$.

\begin{lemme}\label{lem_Z_0_chern_class} We have the following equality in ${\rm CH}_{n-1}(X^2\times Y\backslash \gamma)$ $$Z=i_{V,*}(c_{n-r}(F_{|V}))$$ where $i_V:V\hookrightarrow X^2\times Y\backslash \gamma$.
\end{lemme}
\begin{proof} The equation defining $X\subset Y$ gives rise to a section $\eta_{X,Y}\in H^0(F'_{|V})$, thus to a section $\overline{\eta_{X,Y}}\in H^0(F_{|V})$. As $F$ is globally generated, a general section is regular so for $X$ general, $Z$ which is, by construction, the zero locus of $\overline{\eta_{X,Y}}$ is $n-1$-dimensional and its class is the top Chern class of $F_{|V}$.
\end{proof}

Using the refined Gysin morphisms (see \cite{fulton}), we get the following relation
\begin{proposition}\label{prop_relation_without_diag} There is a homogeneous polynomial $P\in \mathbb Q[T_1,T_2,T_3]$ of degree $2n-1$ such that $$\begin{aligned} Z+j_{12,*}(c_{n-r}(F_{|\Delta_{12}})c_r(Exc(\Delta_{12}\backslash V)))+j_{13,*}(c_{n-r}(F_{|\Gamma_{13}})c_{r-1}(Exc(\Gamma_{13}\backslash V))) \\+j_{23,*}(c_{n-r}(F_{|\Gamma_{23}})c_{r-1}(Exc(\Gamma_{23}\backslash V))) + P(h_{X,1},h_{X,2},h_{Y})=0\end{aligned}$$ in ${\rm CH}_{n-1}(X^2\times Y\backslash \gamma)$, where $h_{X,i}=pr_i^{X^2\times Y,*}(H_{|X})$, $h_Y=pr_3^{X^2\times Y,*}(H_{|Y})$ and $$j_{12}:X\times Y\backslash \Gamma\hookrightarrow X^2\times Y\backslash \gamma,\ (x,y)\mapsto (x,x,y)$$
$$j_{13}:X\times X\backslash \Delta_X\hookrightarrow X^2\times Y\backslash \gamma,\ (x,x')\mapsto (x,x',x)$$
$$j_{23}:X\times X\backslash \Delta_X\hookrightarrow X^2\times Y\backslash \gamma,\ (x,x')\mapsto (x,x',x')$$
\end{proposition}
\begin{proof} The proof is readily adapted from \cite[Proposition 1.7]{fu_decomp}. Consider the following fiber square $$\xymatrix{V\cup\Delta_{12}\cup\Gamma_{13}\cup\Gamma_{23}\ar@{^{(}->}[r]^{i_3}\ar@{^{(}->}[d]^{i_4} &W\ar@{^{(}->}[d]^{i_2}\\
X^2\times Y\backslash\gamma\ar@{^{(}->}[r]_{i_1} &(\mathbb P^{n+r})^3\backslash \delta_{\mathbb P^{n+r}}}$$
According to \cite[Theorem 6.2]{fulton}, we have in ${\rm CH}_{n-1}(X^2\times Y\backslash \gamma)$, 
\begin{equation}\label{equality_gysin_commute} i_{4,*}(i_1^{!}c_{n-r}(F))=i_1^!(i_{2,*}c_{n-r}(F))
\end{equation}
\indent As $i_{2,*}c_{n-r}(F)\in {\rm CH}_{n+3r+1}((\mathbb P^{n+r})^3\backslash \delta_{\mathbb P^{n+r}})\simeq {\rm CH}_{n+3r+1}((\mathbb P^{n+r})^3)$ 
(the latter isomorphism is by localization exact sequence as $dim(\delta_{\mathbb P^{n+r}})=n+r<n+3r+1$) and the Chow ring 
of $(\mathbb P^{n+r})^3$ is generated by $H_i:=pr_i^*H$, $i=1,2,3$, there is a homogeneous polynomial $P$ of degree $3(n+r)-(n+3r+1)=2n-1$ 
such that $i_{2,*}c_{n-r}(F)=-P(H_1,H_2,H_3)$ in ${\rm CH}_{n+3r+1}((\mathbb P^{n+r})^3\backslash \delta_{\mathbb P^{n+r}})$. 
As a result, (\ref{equality_gysin_commute}) becomes $i_{4,*}(i_1^!c_{n-r}(F))+P(h_{X,1},h_{X,2},h_Y)=0$ in ${\rm CH}_{n-1}(X^2\times Y\backslash \gamma)$.\\
\indent Now,according to \cite[Proposition 6.3]{fulton}, we have in ${\rm CH}_{n-1}(W\cap (X^2\times Y\backslash \gamma)$: $$i_1^!c_{n-r}(F)=i_1^!(c_{n-r}(F)\cdot [W])=c_{n-r}(F_{|V\cup \Delta_{12}\cup \Gamma_{13}\cup \Gamma_{23}})i_1^!([W])$$ where $[W]$ denotes the fundamental class of $W$. As $V\cap \Delta_{12}=A_{12}$, $V\cap \Gamma_{i3}=B_{i3}$ ($i=1,2$) have dimension strictly less than $dim(V)$, we get $$\begin{aligned}i_1^!([W])&=[V]+[\Delta_{12}]\cdot c_r(Exc(\Delta_{12}\backslash V)) + [\Gamma_{13}]\cdot c_{r-1}(Exc(\Gamma_{13}\backslash V)) + [\Gamma_{23}]\cdot c_{r-1}(Exc(\Gamma_{23}\backslash V))\end{aligned}$$ and $$c_{n-r}(F_{|V})i_1^![W]=c_{n-r}(F_{|V}),$$ $$c_{n-r}(F_{|\Delta_{12}})\cdot i_1^![W]=j_{12,*}(c_{n-r}(F_{|\Delta_{12}})\cdot c_r(Exc(\Delta_{12}\backslash V))),$$ and $$c_{n-r}(F_{|\Gamma_{i3}})\cdot i_1^![W]=j_{i3,*}(c_{n-r}(F_{|\Gamma_{i3}})\cdot c_{r-1}(Exc(\Gamma_{i3}\backslash V))).$$
\end{proof}

Let us calculate more precisely each term of the relation of the Proposition.

\begin{lemme}\label{lem_restrictions_F} (1) Under the isomorphism $\Delta_{12}\simeq X\times Y\backslash \Gamma$, we have $$F_{|\Delta_{12}}\simeq \oplus_{m=d-(n-r+1)}^{d-2}pr_X^*\mathcal O_X(m)\otimes pr_Y^*\mathcal O_Y(d-m)$$ and $$c_{n-r}(F_{|\Delta_{12}})=\Pi_{m= d-(n-r+1)}^{d-2}(m h_X+ (d-m)h_Y).$$

\indent (2) Under the isomorphism $\Gamma_{13}\simeq X\times X\backslash \Delta_X$, we have $$F_{|\Gamma_{13}}\simeq \oplus_{m=d-(n-r)}^{d-1}pr_1^*\mathcal O_X(m)\otimes pr_2^*\mathcal O_X(d-m)$$ and $$c_{n-r}(F_{|\Gamma_{13}})=\Pi_{m=d-(n-r)}^{d-1}(m pr_1^*h_X+(d-m)pr_2^*h_X).$$  
\end{lemme}
\begin{proof} (1) We have the inclusions $$\xymatrix{\mathbb P(t^*\mathcal E_{2|X\times Y\backslash \Gamma})\ar[d]\ar@{^{(}->}[r] &\mathbb P(t^*\mathcal E_{2|(\mathbb P^{n+r})^2\backslash \Delta_{\mathbb P^{n+r}}})\ar@{^{(}->}[r]\ar[d]^{p_2} &\mathbb P(t^*\mathcal E_2)\ar[d]_{p_W}\\
X\times Y\backslash \Gamma\ar@{^{(}->}[r] &(\mathbb P^{n+r})^2\backslash \Delta_{\mathbb P^{n+r}}\ar@{^{(}->}[r] &W}$$ which shows that $F_{|\Delta_{12}}'$, $F_{|\Delta_{12}}''$, $F_{|\Delta_{12}}$ are pulled-back from $(\mathbb P^{n+r})^2\backslash \Delta_{\mathbb P^{n+r}}$. The three sections $D_{W,i}$, $i=1,2,3$ of $p_W$ restrict to two sections $p_2$, namely $D_1:(x,y)\mapsto ([\ell_{x,y}],x)$ (restriction of $D_{W,1}$ and $D_{W,2}$) and $D_2:(x,y)\mapsto ([\ell_{x,y}],y)$ (restriction of $D_{W,3}$).\\
\indent Evaluation maps give an isomorphism $t^*\mathcal E_{2|(\mathbb P^{n+r})^2\backslash \Delta_{\mathbb P^{n+r}}}\simeq pr_1^*\mathcal O_{\mathbb P^{n+r}}(1)\oplus pr_2^*\mathcal O_{\mathbb P^{n+r}}(1)$. The section $D_1$ is given by the surjection $t^*\mathcal E_{2|(\mathbb P^{n+r})^2\backslash \Delta_{\mathbb P^{n+r}}}\twoheadrightarrow pr_1^*\mathcal O_{\mathbb P^{n+r}}(1)$ so that $D_1\in |p_2^*pr_2^*\mathcal O_{\mathbb P^{n+r}}(-1)\otimes q_2^*\mathcal O_{\mathbb P^{n+r}}(1)|$ where $q_2:\mathbb P(\mathcal E_{2|(\mathbb P^{n+r})^2\backslash \Delta_{\mathbb P^{n+r}}})\rightarrow \mathbb P^{n+r}$.\\
\indent Likewise $D_2\in |p_2^*pr_1^*\mathcal O_{\mathbb P^{n+r}}(-1)\otimes q_2^*\mathcal O_{\mathbb P^{n+r}}(1)|$.\\
\indent So we have $$\begin{aligned}F'_{|(\mathbb P^{n+r})^2\backslash \Delta_{\mathbb P^{n+r}}}&\simeq p_{2,*}(q_2^*\mathcal O_{\mathbb P^{n+r}}(d)\otimes \mathcal O_{\mathbb P(\mathcal E_{2|(\mathbb P^{n+r})^2\backslash \Delta_{\mathbb P^{n+r}}})}(-2D_1))\\
&\simeq p_{2,*}(q_2^*\mathcal O_{\mathbb P^{n+r}}(d-2)\otimes p_2^*pr_2^*\mathcal O_{\mathbb P^{n+r}}(2))\\
&\simeq {\rm Sym}^{d-2}t^*\mathcal E_{2|(\mathbb P^{n+r})^2\backslash \Delta_{\mathbb P^{n+r}}}\otimes pr_2^*\mathcal O_{\mathbb P^{n+r}}(2)\\
&\simeq \oplus_{m=0}^{d-2}pr_1^*\mathcal O_{\mathbb P^{n+r}}(m)\otimes pr_2^*\mathcal O_{\mathbb P^{n+r}}(d-2-m+2)\end{aligned}$$ and $$\begin{aligned}F''_{|(\mathbb P^{n+r})^2\backslash \Delta_{\mathbb P^{n+r}}}&\simeq p_{2,*}(q_2^*\mathcal O_{\mathbb P^{n+r}}(d)\otimes \mathcal O_{\mathbb P(\mathcal E_{2|(\mathbb P^{n+r})^2\backslash \Delta_{\mathbb P^{n+r}}})}(-(n-r+2)D_1)\\
&\simeq {\rm Sym}^{d-(n-r+2)}t^*\mathcal E_{2|(\mathbb P^{n+r})^2\backslash \Delta_{\mathbb P^{n+r}}}\otimes pr_2^*\mathcal O_{\mathbb P^{n+r}}(n-r+2)\\
&\simeq \oplus_{m=0}^{d-(n-r+2)}pr_1^*\mathcal O_{\mathbb P^{n+r}}(m)\otimes pr_2^*\mathcal O_{\mathbb P^{n+r}}(d-m).\end{aligned}$$
As a consequence, $F_{|(\mathbb P^{n+r})^2\backslash \Delta_{\mathbb P^{n+r}}}\simeq \oplus_{m=d-(n-r+1)}^{d-2}pr_1^*\mathcal O_{\mathbb P^{n+r}}(m)\otimes pr_2^*\mathcal O_{\mathbb P^{n+r}}(d-m)$.\\

\indent (2) Consider the isomorphism $\Gamma_{13}\simeq X\times X\backslash \Delta_X$, $(x,x',x)\mapsto (x,x')$ and the inclusions $$\xymatrix{\mathbb P(t^*\mathcal E_{2|X^2\backslash \Delta_X})\ar[d]\ar@{^{(}->}[r] &\mathbb P(t^*\mathcal E_{2|(\mathbb P^{n+r})^2\backslash \Delta_{\mathbb P^{n+r}}})\ar@{^{(}->}[r]\ar[d]^{p_2} &\mathbb P(t^*\mathcal E_2)\ar[d]_{p_W}\\
X\times X\backslash \Delta_X\ar@{^{(}->}[r] &(\mathbb P^{n+r})^2\backslash \Delta_{\mathbb P^{n+r}}\ar@{^{(}->}[r] &W}.$$
The three sections $D_{W,i}$, $i=1,2,3$ of $p_W$ restrict to two sections of $p_2$, namely $D_1:(x,x')\mapsto ([\ell_{x,x'}],x)$ (restriction of $D_{W,1}$ and $D_{W,3}$) and $D_2:(x,x')\mapsto ([\ell_{x,x'}],x')$ (restriction of $D_{W,2}$).\\
\indent Evaluation maps give the isomorphism $t^*\mathcal E_{2|(\mathbb P^{n+r})^2\backslash \Delta_{\mathbb P^{n+r}}}\simeq pr_1^*\mathcal O_{\mathbb P^{n+r}}(1)\oplus pr_2^*\mathcal O_{\mathbb P^{n+r}}(1)$. We get, as above, that $D_1\in |p_2^*pr_2^*\mathcal O_{\mathbb P^{n+r}}(-1)\otimes q_2^*\mathcal O_{\mathbb P^{n+r}}(1)|$ and $D_2\in |p_2^*pr_1^*\mathcal O_{\mathbb P^{n+r}}(-1)\otimes q_2^*\mathcal O_{\mathbb P^{n+r}}(1)|$.\\
\indent Hence $$\begin{aligned}F'_{|(\mathbb P^{n+r})^2\backslash \Delta_{\mathbb P^{n+r}}}&\simeq p_{2,*}(q_2^*\mathcal O_{\mathbb P^{n+r}}(d)\otimes \mathcal O_{\mathbb P(\mathcal E_{2|(\mathbb P^{n+r})^2\backslash \Delta_{\mathbb P^{n+r}}})}(-D_1-D_2)\\
&\simeq {\rm Sym}^{d-2}t^*\mathcal E_{2|(\mathbb P^{n+r})^2\backslash \Delta_{\mathbb P^{n+r}}}\otimes pr_1^*\mathcal O_{\mathbb P^{n+r}}(1)\otimes pr_2^*\mathcal O_{\mathbb P^{n+r}}(1)\\
&\simeq \oplus_{k=1}^{d-1}pr_1^*\mathcal O_{\mathbb P^{n+r}}(k)\otimes pr_2^*\mathcal O_{\mathbb P^{n+r}}(d-k)\end{aligned}$$ and $$\begin{aligned}F''_{|(\mathbb P^{n+r})^2\backslash \Delta_{\mathbb P^{n+r}}}&\simeq p_{2,*}(q_2^*\mathcal O_{\mathbb P^{n+r}}(d)\otimes \mathcal O_{\mathbb P(\mathcal E_{2|(\mathbb P^{n+r})^2\backslash \Delta_{\mathbb P^{n+r}}})}(-(n-r+1)D_1-D_2)\\
&\simeq {\rm Sym}^{d-(n-r+2)}t^*\mathcal E_{2|(\mathbb P^{n+r})^2\backslash \Delta_{\mathbb P^{n+r}}}\otimes pr_1^*\mathcal O_{\mathbb P^{n+r}}(1)\otimes pr_2^*\mathcal O_{\mathbb P^{n+r}}(n-r+1)\\
&\simeq \oplus_{k=1}^{d-(n-r+1)}pr_1^*\mathcal O_{\mathbb P^{n+r}}(k)\otimes pr_2^*\mathcal O_{\mathbb P^{n+r}}(d-k).\end{aligned}$$
As a result $F_{|(\mathbb P^{n+r})^2\backslash \Delta_{\mathbb P^{n+r}}}\simeq \oplus_{k=d-(n-r)}^{d-1}pr_1^*\mathcal O_{\mathbb P^{n+r}}(k)\otimes pr_2^*\mathcal O_{\mathbb P^{n+r}}(d-k)$  
\end{proof}

Putting everything together, we get
\begin{lemme}\label{lem_computing_Q} (1) There is a homogeneous polynomial $Q_1\in \mathbb Q(T_1,T_2)$ of degree $n$ such that $Q_1(h_X,h_Y)=c_{n-r}(F_{|\Delta_{12}})\cdot c_r(\frac{pr_X^*N_{X/\mathbb P^{n+r}}}{pr_Y^*\mathcal O_Y(-1)\otimes pr_X^*\mathcal O_X(1)})$. Writing $Q(h_X,h_Y)=\sum_{i=0}^na_{i,Q_1}h_X^ih_Y^{n-i}$ we have $a_{0,Q_1}=(n-r+1)!$.\\
\indent (2) There is a homogeneous polynomial $Q_2\in \mathbb Q[T_1,T_2]$ of degree $n-1$ such that $Q_2(h_{X,1},h_{X,2})=c_{n-r}(F_{|\Gamma_{13}})\cdot c_{r-1}(Exc(\Gamma_{13}\backslash V)$. Writing $Q_2(h_{X,1},h_{X,2})=\sum_{i=0}^{n-1}a_{i,Q_2}h_{X,1}^ih_{X,2}^{n-1-i}$, we have $a_{0,Q_2}=(n-r)!$.\\
\indent There is a homogeneous degree $n-1$ polynomial $Q_3=\sum_{i=0}^{n-1}a_{i,Q_3}T_1^iT_2^{n-1-i}\in \mathbb Q[T_1,T_2]$ such that $Q_3(h_{X,1},h_{X,2})=c_{n-r}(F_{|\Gamma_{23}})\cdot c_{r-1}(Exc(\Gamma_{23}\backslash V)$.
\end{lemme}
\begin{proof} (1) According to Proposition \ref{prop_intersect_V_diagonals_excess_bundles} (2), 
$$\begin{aligned}c(Exc(\Delta_{12}\backslash V))&=\frac{(1+dh_X)\Pi_{m=1}^r(1+d_mh_X)}{1-(h_Y-h_X)}\\
&=(\sum_{m=0}^{r+1}b_mh_X^m)\cdot \sum_{k\geq 0}(h_Y-h_X)^k\ {\rm for\ some}\ b_m\\
& \ \ \ \ \ {\rm determined\ by\ the}\ d_i{\rm 's\ and}\ d\\
&=\sum_{k\geq 0}\sum_{m=0}^{r+1}\sum_{p=0}^k \binom k p (-1)^pb_mh_X^{p+m}h_Y^{k-p}\end{aligned}$$
so that $c_r(Exc(\Delta_{12}\backslash V))=\sum_{m=0}^{r+1}\sum_{k=0}^{r-m}\binom{r-m}{k}(-1)^k b_mh_X^{k+m}h_Y^{r-m-k}$.\\
\indent The coefficient of $h_Y^n$ in $Q_1(h_X,h_Y)$ is the product of ${\rm coeff}(h_Y^{n-r},c_{n-r}(F_{|\Delta_{12}})=\Pi_{m=d-(n-r+1)}^{d-2}(d-m)=(n-r+1)!$ and ${\rm coeff}(h_Y^r,c_r(Exc(\Delta_{12}\backslash V)))=\binom{r}{0}b_0=1$.\\

\indent (2) According to Proposition \ref{prop_intersect_V_diagonals_excess_bundles} (3), 
$$\begin{aligned}c(Exc(\Gamma_{13}\backslash V))&= \frac{\Pi_{i=1}^r(1+d_ih_{X,1})}{1-(h_{X,2}-h_{X,1})}\\
&= (\sum_{m=0}^{r}b'_mh_{X_1}^m)\cdot \sum_{k\geq 0}(h_{X_2}-h_{X,1})^k \ {\rm for\ some}\ b'_m\\
& \ \ \ \ \ {\rm determined\ by\ the}\ d_i{\rm 's\ and}\ d\\
&=\sum_{k\geq 0}\sum_{m=0}^r\sum_{p=0}^k\binom{k}{p} (-1)^pb'_mX_{X,1}^{p+m}h_{X,2}^{k-p}\end{aligned}$$
so that $c_{r-1}(Exc(\Gamma_{13}\backslash V))=\sum_{m=0}^r\sum_{k=0}^{r-1-m}\binom{r-1-m}{k}(-1)^k b'_m h_{X,1}^{k+m}h_{X,2}^{r-1-m-k}$.\\
\indent The coefficient of $h_{X,2}^{n-1}$ in $Q_2(h_{X,1},h_{X,2})$ is the product of ${\rm coeff}(h_{X,2}^{n-r},c_{n-r}(F_{|\Gamma_{13}}))=\Pi_{m=d-(n-r)}^{d-1}(d-m)=(n-r)!$ and ${\rm coeff}(h_{X,2}^{r-1},c_{r-1}(Exc(\Gamma_{13}\backslash V)))=\binom{r-1}{0}b'_0=1$.

\end{proof}

\begin{theoreme}\label{thm_identity_X_X_Y} We have the following equality in ${\rm CH}_{n-1}(X^2\times Y)$:
$$\begin{aligned}
((n-r+1)!+(n-r)!)\gamma = Z + j_{12,*}(Q_1(h_X,h_Y))+j_{13,*}(Q_2(h_{X,1},h_{X,2}))
+j_{23,*}(Q_3(h_{X,1},h_{X,2})) \\ 
+ P(h_{X,1},h_{X,2},h_Y).
\end{aligned}$$
\end{theoreme}
\begin{proof} Putting Proposition \ref{prop_relation_without_diag}, Lemma \ref{lem_computing_Q} and localization of the exact sequence together, we get, for some $N\in \mathbb Q$, the following equality in ${\rm CH}_{n-1}(X^2\times Y)$
\begin{equation}\label{identity_N_to_compute}
N\gamma = Z + j_{12,*}(Q_1(h_X,h_Y))+j_{13,*}(Q_2(h_{X,1},h_{X,2}))+j_{23,*}(Q_3(h_{X,1},h_{X,2})) + P(h_{X,1},h_{X,2},h_Y)
\end{equation}
By the adjunction formula, $K_X\simeq \mathcal O_X(k)$ with $k>n+r$, so that $h^0(K_X)>0$. By Roitman theorem (\cite{vois_lect}), ${\rm CH}_0(X)$ is thus infinite dimensional. In particular there are $0$-cycle on $X$ that are not multiples of $h_X^{n-1}$. Pick such $0$-cycle $z\in {\rm CH}_0(X)_{hom}$.\\
\indent Looking at the correspondences in (\ref{identity_N_to_compute}) as correspondences from $X_1\times Y$ to $X_2$ ($X_i$ stands a copy of $X$ at the $i$-component), we can let (\ref{identity_N_to_compute}) act on $z\times Y$.\\
\indent We have $\gamma_*(z\times Y)=z\in {\rm CH}_0(X)$.\\
\indent We have in ${\rm CH}_0(X)$, $$\begin{aligned} j_{12,*}(Q_1(h_X,h_Y))_*(z\times Y)&=pr^{X^2\times Y}_{2,*}(j_{12,*}(Q_1(h_X,h_Y))\cdot z\times X\times Y)\\
&=pr_{2,*}^{X^2\times Y}j_{12,*}(Q_1(h_X,h_Y)\cdot j_{12}^*(z\times X\times Y))\\
&=pr_{2,*}^{X^2\times Y}j_{12,*}(Q_1(h_X,h_Y)\cdot z\times Y)\\
&=pr_{2,*}^{X^2\times Y}j_{12,*}(a_{0,Q_1}z\cdot h_Y^n)\\
&=(n-r+1)!deg(Y)z\ {\rm by\ Lemma}\ \ref{lem_computing_Q}.\end{aligned}$$ 
\indent We easily see that $j_{13,*}(Q_2(h_{X,1},h_{X,2}))_*(z\times Y)\in \mathbb Q\cdot h_X^{n-1}$, $j_{23,*}(Q_3(h_{X,1},h_{X,2}))_*(z\times Y)\in \mathbb Q\cdot h_X^{n-1}$ and $P(h_{X,1},h_{X,2},h_Y)_*(z\times Y)\in \mathbb Q\cdot h_X^{n-1}$.\\
\indent We are thus left with the analysis of $Z_*(z\times Y)$.
We have the following 
\begin{lemme}\label{lem_action_of_Z_on_points} Let $x\in X$, then $Z_*(x\times Y)=((n-r+1)!(1-{\rm deg}(Y))+(n-r)!)x\ {\rm mod}\ \mathbb Q\cdot h_X^{n-1}$.
\end{lemme}
Postponing the proof of the Lemma, we conclude as follows: the action of (\ref{identity_N_to_compute}) on $z\times Y$ gives $Nz=((n-r+1)!(1-{\rm deg}(Y))+(n-r)!)z + (n-r+1)!{\rm deg}(Y)z \ {\rm mod}\ h_X^{n-1}$ in ${\rm CH}_0(X)$ i.e. $N=(n-r+1)!+(n-r)!\neq 0$.

\end{proof}

\begin{proof}[Proof of Lemma \ref{lem_action_of_Z_on_points}] We recall that in the classical construction, associating to a pair of distinct points the line they determine 
\begin{equation}\label{diag_PxP_G}\xymatrix{\widetilde{(\mathbb P^{n+r})^2_{\Delta_{\mathbb P^{n+r}}}}\ar[d]_b\ar[dr]^{e}\ar[r]^{\cong} &\mathbb P(\mathcal E_2)\times_G\mathbb P(\mathcal E_2)\ar[d]\\
(\mathbb P^{n+r})^2\ar@{-->}[r]^{\tilde e} &G(2,n+r+1)}
\end{equation} the exceptional divisor $E_{\Delta}$ of $b$ satisfies $E_\Delta=b^*(H_1+H_2)-e^*g$ in ${\rm Pic}(\widetilde{(\mathbb P^{n+r})^2_{\Delta_{\mathbb P^{n+r}}}})$ where $g$ denote the Pl\"ucker line bundle.\\
\indent Restricting the situation to $x\times Y$ and completing the picture, we get $$\xymatrix{\mathbb P(e^*\mathcal E_2)\ar[r]^{q_Y}\ar[d]_{p_Y} &\mathbb P^{n+r}\\
\widetilde{Y_x}\ar[dr]^e\ar[d]_b &\\
x\times Y\ar@{-->}[r] &G(2,n+r+1)}$$
with the exceptonal divisor $E_x\simeq \mathbb P(\Omega_{Y,x})$ of (abuse of notation) $b$ satisfying $E_x=b^*h_Y-e^*g$ in ${\rm Pic}(\widetilde{Y_x})$. The projection $p_Y$ admits two sections. The section $\sigma_x:(x,y)\mapsto ([\ell_{x,y}],x)$, comes from the surjection (evaluation of linear form on the line) $e^*\mathcal E_2\twoheadrightarrow \mathcal O_{\widetilde{Y_x}}$. As $e^*g=c_1(e^*\mathcal E_2)=b^*h_Y-E_x$, the following sequence is exact 
\begin{equation}\label{exact_seq_E_2_Y_x} 0\rightarrow b^*\mathcal O_Y(1)\otimes \mathcal O_{\widetilde{Y_x}}(-E)\rightarrow e^*\mathcal E_2\rightarrow \mathcal O_{\widetilde{Y_x}}\rightarrow 0
\end{equation} so that $\sigma_x(\widetilde{Y_x})=D_x\in |p_Y^*\mathcal K_x^\vee\otimes q_Y^*\mathcal O_{\mathbb P^{n+r}}(1)|$ where $\mathcal K_x\simeq b^*\mathcal O_Y(1)\otimes \mathcal O_{\widetilde{Y_x}}(-E)$.\\
\indent To establish a relation to $Z$ (\ref{def_Z_0}), we have to focus on $$S_x:=\overline{\{(x,y)\in x\times Y\backslash \{(x,x)\}, P_{d|\ell_{x,y}}\in L_x^{n-r+1}H^0(\mathcal O_{\ell_{x,y}}(d-(n-r+1))\}}\overset{i_{S_x}}{\subset} \widetilde{Y_x}$$
\indent The inclusion $q_Y^*\mathcal O_{\mathbb P^{n+r}}(d)\otimes \mathcal O_{\mathbb P(e^*\mathcal E_2)}(-(n-r+1)D_x)\overset{\sigma_x^{(n-r)}}{\hookrightarrow} q_Y^*\mathcal O_{\mathbb P^{n+r}}(d)\otimes \mathcal O_{\mathbb P(e^*\mathcal E_2)}(-D_x)$ gives rise to an exact sequence $$0\rightarrow {\rm Sym}^{d-(n-r+1)}e^*\mathcal E_2\otimes \mathcal K_x^{(n-r+1)}\rightarrow {\rm Sym}^{d-1}e^*\mathcal E_2\otimes \mathcal K_x\rightarrow \mathcal F\rightarrow 0$$ where we can check that $\mathcal F$ is a rank $n-r$ vector bundle.\\
\indent ``The'' equation $P_d$ of $X\subset Y$ (i.e. the choice of a polynomial which restricts to a section of $\mathcal O_Y(d)$ defining $X$) gives rise to a section of ${\rm Sym}^{d-1}e^*\mathcal E_2\otimes \mathcal K_x$ thus, by projection to a section of $\mathcal F$. By construction, $S_x$ is the zero locus of that section. So for $X$ general in $|\mathcal O_Y(d)|$, $S_x$ is a smooth, $r$-dimensional and $[S_x]=c_{n-r}(\mathcal F)$ in ${\rm CH}_r(\widetilde{Y_x})$.\\

\indent We also need to introduce $\Sigma_x:=\widetilde{X_x}\cap S_x$ where $\widetilde{X_x}$ is the strict transform of $x\times X\subset x\times Y$ under $b$. We have $\widetilde{X_x}\in |b^*\mathcal O_Y(d)\otimes \mathcal O_{\widetilde{Y_x}}(-E)|$ and we can introduce the notations $x\times X\overset{b_X}{\leftarrow}\widetilde{X_x}\overset{p_X}{\leftarrow}\mathbb P(e^*\mathcal E_{2|\widetilde{X_x}})\overset{q_X}{\rightarrow} \mathbb P^{n+r}$ and we will also denote (abuse of notation) $E_x\simeq \mathbb P(\Omega_{X,x})\subset \widetilde{X_x}$ the exceptional divisor.\\
\indent The projective bundle $p_S:\mathbb P(i_{S_x}^*e^*\mathcal E_2)\rightarrow S_x$ (restriction of $p_Y$) admits also the section $\bar\sigma_x:s\mapsto (e(i_{S_x}(s)), x)$.\\
\indent The section $(x,y)\mapsto ([\ell_{x,y}],y)$ of $p_Y$ restricts to a section $\bar\sigma_X:(x,x')\mapsto ([\ell_{x,x'}],x')$ of $p_\Sigma\mathbb P((e^*\mathcal E_2)_{|\Sigma_x})\rightarrow \Sigma_x$ and it is defined by the surjection $(e^*\mathcal E_2)_{|\Sigma_x}\twoheadrightarrow (b^*\mathcal O_Y(1))_{|\Sigma_x}$; so that $\bar\sigma_X(\Sigma_x)\in |(p_X^*\mathcal O_{\widetilde X}(E)\otimes q_X^*\mathcal O_{\mathbb P^{n+r}}(1))_{|\Sigma_x}|$ as a divisor of $\mathbb P((e^*\mathcal E_2)_{|\Sigma_x})$.\\
\indent By construction $q_S^{-1}(X)=\bar\sigma_x(S_x)\cup \bar\sigma_X(\Sigma_x)\cup R$ where $q_S:\mathbb P(i_{S_x}^*e^*\mathcal E_2)\rightarrow \mathbb P^{n+r}$ and $R$ is a $0$-dimensional subscheme such that $\overline{q_{S}}_*(R)=Z_*(x\times Y)$ where $\overline{q_S}:q_S^{-1}(X)\rightarrow X$. The component $\bar\sigma_x(S_x)$ have excess dimension $r$ and generic multiplicty $n-r+1$ and $\bar\sigma_X(\Sigma_x)$ have excess dimension $r-1$ and generic multiplicity $1$.\\
\indent The excess normal sheaf $Exc(\bar\sigma_x(S_x)\backslash R)$ along $\bar\sigma_x(S_x)\backslash R$ is isomorphic to $\overline{q_S}^*N_{X/\mathbb P^{n+r}}/(p_S^*\mathcal K_x^\vee\otimes q_S^*\mathcal O_{\mathbb P^{n+r}}(1))_{|\bar\sigma_x(S_x)}$.\\
\indent By the exact sequence defining the section $\sigma_x$, we have $\bar\sigma_x^*q_S^*\mathcal O_{\mathbb P^{n+r}}(1)\simeq \mathcal O_{S_x}$. Moreover $(\overline{q_S}^*N_{X/\mathbb P^{n+r}})_{|\bar\sigma_x(S_x)}\simeq N_{X/\mathbb P^{n+r},x}\otimes \mathcal O_{S_x}$. As a result 
\begin{equation}\label{computation_c_r_Exc_sigma_x}
\begin{aligned}c_r(Exc(\bar\sigma_x(S_x)\backslash R))&=c_r(\frac{O_{S_x}^{r+1}}{\mathcal K_{x|S_x}^\vee})\\
&= c_1(\mathcal K_{x|S_x})^r\\
&=i_{S_x}^*(b^*h_Y-E_x)^r\\
&=i_{S_x}^*\sum_{k=0}^r\binom{r}{k}(-E_x)^kb^*h_Y^{r-k}\\
&=i_{S_x}^*[b^*h_Y^r -\sum_{k=1}^r\binom{r}{k}j_{E_x,*}((-E_{x|E_x})^{k-1}b_{|E_x}^*\underset{=0\ {\rm for}\ r-k>0}{i_x^*h_Y^{r-k}}]\\
&=i_{S_x}^*(b^*h_Y^r -j_{E_x,*}(-E_{x|E_x})^{r-1})
\end{aligned}
\end{equation}

The excess normal sheaf $Exc(\bar\sigma_X(\Sigma_x)\backslash R)$ along $\bar\sigma_X(\Sigma_x)\backslash R$ is isomorphic to $\overline{q_S}^*N_{X/\mathbb P^{n+r}}/N_{\bar\sigma_X(\Sigma_x)/\mathbb P(i_S^*e^*\mathcal E_2)}$.\\
\indent The bundle $N_{\bar\sigma_X(\Sigma_x)/\mathbb P(i_S^*e^*\mathcal E_2)}$ sits in the exact sequence 
$$0\rightarrow N_{\bar\sigma_X(\Sigma_x)/\mathbb P((e^*\mathcal E_2)_{|\Sigma_x})}\rightarrow N_{\bar\sigma_X(\Sigma_x)/\mathbb P(i_S^*e^*\mathcal E_2)}\rightarrow N_{(\mathbb P((e^*\mathcal E_2)_{|\Sigma_x})/\mathbb P(i_S^*e^*\mathcal E_2)|\bar\sigma_X(\Sigma_x)}\rightarrow 0.$$
\indent By the exact sequence defining the section $\sigma_X$, we have $N_{\bar\sigma_X(\Sigma_x)/\mathbb P((e^*\mathcal E_2)_{|\Sigma_x})}\simeq \sigma_X^*(p_X^*\mathcal O_{\widetilde{X_x}}(E_x)\otimes q_X^*\mathcal O_{\mathbb P^{n+r}}(1))_{|\mathbb P(e^*\mathcal E_{2|\Sigma_x})})\simeq (\mathcal O_{\widetilde{X_x}}(E_x)\otimes b_X^*\mathcal O_X(1))_{|\Sigma_x}$. We have $N_{(\mathbb P((e^*\mathcal E_2)_{|\Sigma_x})/\mathbb P(i_S^*e^*\mathcal E_2)|\bar\sigma_X(\Sigma_x)}\simeq (b_X^*\mathcal O_X(d)\otimes \mathcal O_{\widetilde{X_x}}(-E_x))_{|\Sigma_x}$.\\
\indent As a result 
$$\begin{aligned}
c(Exc(\bar\sigma_X(\Sigma_x)\backslash R))&=i_{\Sigma}^*(((1+db_X^*h_X)\Pi_{i=1}^r(1+d_ib_X^*h_X))(1+(E_x+b_X^*h_X))^{-1}(1-(E_x-db_X^*h_X))^{-1})\\
&=i_{\Sigma}^*((\sum_{i=0}^{r+1}\lambda_ib_X^*h_X^i)\cdot \sum_{m\geq 0}(-1)^m(E_x+b_X^*h_X)^m\sum_{k\geq 0}(E_x-db_X^*h_X)^k)\\ 
& \ \ \ \ \  {\rm where\  the}\ \lambda_i\ {\rm are\ determined\ by\ the}\ d_i{\rm 's\ and}\ d\\
&=i_{\Sigma}^*\sum_{m\geq 0}\sum_{k\geq 0}\sum_{i=0}^{r+1}(-1)^m\lambda_i(E_x+b_X^*h_X)^m(E_x-db_X^*h_X)^kh_X^i)
\end{aligned}$$
so that
\begin{equation}\label{computation_Exc_sigma_X}
\begin{aligned} 
c_{r-1}(Exc(\bar\sigma_X(\Sigma_x)\backslash R))&=\sum_{i=0}^{r+1}\sum_{m=0}^{r-1-i}(-1)^m\lambda_i(E_x+b_X^*h_X)^m(E_x-db_X^*h_X)^{r-1-i-m}h_X^i
\end{aligned}
\end{equation}

\indent Now, we can again use the fact that refined Gysin maps commute with push-forward in the diagram:
$$\xymatrix{q_S^{-1}(X) \ar@{^{(}->}[r]\ar[d]_{\overline{q_S}} &\mathbb P(i_{S_x}^*e^*\mathcal E_2)\ar[d]^{q_S}\\
X\ar@{^{(}->}[r]_{i_X} &\mathbb P^{n+r}}$$
we get 
\begin{equation}\label{equality_commutativity_S_x} \overline{q_S}_*(i_X^![\mathbb P(i_{S_x}^*e^*\mathcal E_2)])=i_X^!(q_{S,*}[\mathbb P(i_{S_x}^*e^*\mathcal E_2)])
\end{equation} in ${\rm CH}_0(X)$. We have in ${\rm CH}_0(q_S^{-1}(X))$, $$i_X^![\mathbb P(i_{S_x}^*e^*\mathcal E_2)]=R+[\bar\sigma_x(S_x)]\cdot c_r(Exc(\bar\sigma_x(S_x)\backslash R))+[\bar\sigma_X(\Sigma_x)]\cdot c_{r-1}(Exc(\bar\sigma_X(\Sigma_x)\backslash R))$$ and $\overline{q_S}_*[\bar\sigma_x(S_x)]\cdot c_r(Exc(\bar\sigma_x(S_x)\backslash R)$ is supported on $\{x\}$. Remembering that the multiplicity along $\bar\sigma_x(S_x)$ is $(n-r+1)$, its degree is:
$(n-r+1)\int_{\widetilde{Y_x}}[S_x]\cdot c_r(Exc(\bar\sigma_x(S_x)\backslash R)=(n-r+1)\int_{\widetilde{Y_x}}c_{n-r}(\mathcal F)\cdot c_1(\mathcal K_x)^r$.\\
\indent Looking at (\ref{exact_seq_E_2_Y_x}), in view of the computation of the Chern classes (i.e. in $K_0(\widetilde{Y_x})$) $e^*\mathcal E_2= \mathcal O_{\widetilde{Y_x}}\oplus \mathcal K_x$. So that $\mathcal K_x\otimes {\rm Sym}^{d-1}e^*\mathcal E_2=\oplus_{i=0}^{d-1}\mathcal K_x^{\otimes (i+1)}=\oplus_{i=1}^d\mathcal K_x^{\otimes i}$. Likewise $\mathcal K_x^{\otimes (n-r+1)}\otimes {\rm Sym}^{d-(n-r+1)}e^*\mathcal E_2=\oplus_{i=n-r+1}^d\mathcal K_x^i$. Hence $\mathcal F=\oplus_{i=1}^{n-r}\mathcal K_x^{\otimes i}$. So we have $$\begin{aligned}c_{n-r}(\mathcal F)&=\Pi_{i=1}^{n-r}i(b^*h_Y-E_x)\\
&=(n-r)!(b^*h_Y-E_x)^{n-r}\\
&=(n-r)!\sum_{k=0}^{n-r}\binom{n-r}{k}(-E_x)^kb^*h_Y^{n-r-k}\\
&=(n-r)![b^*h_Y^{n-r} - \sum_{k=1}^{n-r}j_{E_x,*}((-E_{x|E_x})^{k-1}\cdot b_{|E_x}^*\underset{=0\ {\rm if}\ n-r-k>0}{i_x^*h_Y^{n-r-k}})\\
&=(n-r)![b^*h_Y^{n-r} - j_{E_x,*}(-E_{x|E_x})^{n-r-1}].\end{aligned}$$
Hence $$\begin{aligned}(n-r+1){\rm deg}(\overline{q_S}_*[\bar\sigma_x(S_x)]\cdot c_r(Exc(\bar\sigma_x(S_x)\backslash R))&= \int_{\widetilde{Y_x}}(n-r+1)!(b^*h_Y^{n-r}-j_{E_x,*}(-E_{x|E_x})^{n-r-1})\cdot\\
&\ \  \ \ \ \ \ (b^*h_Y^r-j_{E_x,*}(-E_{x|E_x})^{r-1})\\
&=(n-r+1)!\int_{\widetilde{Y_x}}b^*h_Y^n -j_{E_x,*}((-E_{x|E_x})^{r-1}\cdot b_{|E_x}^*\underbrace{i_x^*h_Y^{n-r}}_{=0,\ n-r>0})\\
 &\ \  -j_{E_x,*}((-E_{x|E_x})^{n-r-1}\cdot b_{|E_x}^*\underbrace{i_x^*h_Y^r}_{=0})\\
  &\ \ + j_{E_x,*}((-E_{x|E_x})^{r-1}j_{E_x}^*j_{E_x,*}(-E_{x|E_x})^{n-r-1})\\
&=(n-r+1)!(deg(Y)-1).
\end{aligned}$$

\indent We now deal with the term $[\bar\sigma_X(\Sigma_x)]\cdot c_{r-1}(Exc(\bar\sigma_X(\Sigma_x)\backslash R))$. Under the identification $\Sigma_x\overset{\bar\sigma_X}{\simeq}\bar\sigma_X(\Sigma_x)$, the morphism $\overline{q_S}_{|\bar\sigma_X(\Sigma_x)}$ coincides with the composition $\Sigma_x\hookrightarrow \widetilde{X_x}\rightarrow x\times X\rightarrow X$.\\
\indent We have $[\Sigma_x]=c_{n-r}(\mathcal F_{|\widetilde{X_x}})$ in ${\rm CH}^{n-r}(\widetilde{X_x})$ so that, using (\ref{computation_Exc_sigma_X}), in ${\rm CH}_0(\widetilde{X_x})$, we get:
$$\begin{aligned}
\Sigma_x\cdot c_{r-1}(Exc(\bar\sigma_X(\Sigma_x)\backslash R))&=c_{n-r}(\mathcal F_{|\widetilde{X_x}})\cdot c_{r-1}(Exc(\bar\sigma_X(\Sigma_x)\backslash R))\\
&=(n-r)!(b_X^*h_X^{n-r} -j_{E_x,*}(-E_{x|E_x})^{n-r-1})\cdot \\
& \ \ \ \sum_{i=0}^{r+1}\sum_{m=0}^{r-1-i-m}(-1)^m\lambda_i(E_x+b_X^*h_X)^m(E_x-db_X^*h_X)^{r-1-i-m}h_X^i
\end{aligned}$$

In this expression, any term supported on the exceptional divisor $E_x\subset \widetilde{X_x}$ vanishes when pushed to $X$ outside the terms of the form $(-E_{x|E_x})^{n-2}\cdot b_{X|E_x}^*\alpha$ and the only term not supported on the exceptional divisor is a multiple of $b_X^*h_X^{n-1}$.\\
\indent So the only term of interest in the above expression is of the multiple of $(E_x)^{n-1}$; it is the product of the maximal $(E)$-power of both $c_{n-r}(\mathcal F_{|\widetilde{X_x}})$ and $c_{r-1}(Exc(\bar\sigma_X(\Sigma_x)\backslash R))$ i.e. $$-(n-r)!j_{E_x,*}(-E_{x|E_x})^{n-r-1}\cdot (-1)^{r-1}\lambda_0E_x^{r-1}=-(n-r)!j_{E_x,*}((-E_{x|E_x})^{n-r+r-2})=-(n-r)!.$$
\indent As a result, $$\overline{q_S}_*(\Sigma_x\cdot c_{r-1}(Exc(\bar\sigma_X(\Sigma_x)\backslash R)))=-(n-r)!x\ {\rm mod}\ h_X^{n-1}.$$
Finally, looking at (\ref{equality_commutativity_S_x}) and using $q_{S,*}([\mathbb P(i_{S_x}^*e^*\mathcal E_2)])\in {\rm CH}_{r+1}(\mathbb P^{n+r})\simeq \mathbb Z[T]/(T^{n+r+1}=0)$, we get the result.
\end{proof}

\section{Proof of the theorem}

We can now prove Theorem \ref{thm_1}. We will actually prove the following stronger result:

\begin{theoreme}\label{thm_principal_1} The image of $\Gamma_*:{\rm CH}_k(X)\rightarrow {\rm CH}_k(Y)$ is contained in $\mathbb Q\cdot h_Y^{n-k}$.
\end{theoreme}
\begin{proof} Let $U\subset X$ be an integral subvariety of $X$ of dimension $w$. Looking at the relation of Theorem \ref{thm_identity_X_X_Y} as a relation between correspondences ${\rm CH}_{n+w-1}(X\times X)\rightarrow {\rm CH}_w(Y)$, we can examine the image of $U\times X$ under these correspondences.\\
\indent We readily see that $P(h_{X,1},h_{X,2},h_Y)_*(U\times X)\in \mathbb Q\cdot h_Y^{n-w}$.\\
\indent We have $$\begin{aligned}(j_{23,*}(Q_3(h_{X,1},h_{X,2})))_*(U\times X)&= pr_{Y,*}^{X^2\times Y}j_{23,*}(Q_3(h_{X,1},h_{X,2})\cdot j_{23}^*(U\times X\times Y))\\
&=pr_{Y,*}^{X^2\times Y}j_{23,*}(Q_3(h_{X,1},h_{X,2})\cdot U\times X)\\
&=pr_{Y,*}^{X^2\times Y}j_{23,*}(\sum_{i=0}^{n-1}a_{i,Q_3}(h_{X,1}^i\cdot U) (h_{X,2}^{n-1-i}))\\
&=a_{w,Q_3}(\int_Uh_{X,1}^w)h_{X,2}^{n-1-w}.\end{aligned}$$

\indent We have $$\begin{aligned} (j_{13,*}(Q_2(h_{X,1},h_{X,2})))_*(U\times X)&=pr_{Y,*}^{X^2\times Y}j_{13,*}(Q_2(h_{X,1},h_{X,2})\cdot j_{13}^*(U\times X\times Y))\\
&=pr_{Y,*}^{X^2\times Y}j_{13,*}(Q_2(h_{X,1},h_{X,2})\cdot U\times X)\\
&=pr_{Y,*}^{X^2\times Y}j_{13,*}(\sum_{i=0}^{n-1}a_{i,Q_2}(h_{X,1}^i\cdot U) (h_{X,2}^{n-1-i}))\\
&=a_{0,Q_2}(\int_Xh_X^{n-1})[U] = (n-r)!deg(X)[U]\ {\rm using\ Lemma\ \ref{lem_computing_Q}}.\end{aligned}$$

\indent We have $$\begin{aligned}(j_{12,*}(Q_2(h_X,h_Y)))_*(U\times X)&=pr_{Y,*}^{X^2\times Y}j_{12,*}(Q_1(h_X,h_Y)\cdot j_{12}^*(U\times X\times Y))\\
&=pr_{Y,*}^{X^2\times Y}j_{12,*}(Q_1(h_X,h_Y)\cdot U\times Y)\\
&=pr_{Y,*}^{X^2\times Y}j_{12,*}(\sum_{i=0}^na_{i,Q_1}(h_X^i\cdot U) (h_Y^{n-i}))\\
&=a_{w,Q_1}(\int_Uh_{X|U}^w)h_Y^{n-w}\end{aligned}$$
 
\indent We have $\gamma_*(U\times X)=\Gamma_*(U)=[U]$ in ${\rm CH}_w(Y)$. We are left with the computation of $Z_*(U\times X)$.

\begin{proposition}\label{prop_computation_Z_(UxX)} We have $Z_*(U\times X)= -[(n-r)!deg(X)-2(n-r+1)!][U]\ {\rm mod}\ \mathbb Q\cdot h_Y^{n-w}$ in ${\rm CH}_w(Y)$.
\end{proposition}

Assuming the Proposition, from Theorem \ref{thm_identity_X_X_Y}, we get $$((n-r+1)!-(n-r)!)[U]=(n-r)!(n-r)[U]\in \mathbb Q\cdot h_Y^{n-w}\ {\rm in\ CH}_w(Y)$$ proving the Theorem.
\end{proof}

\begin{proof}[Proof of Proposition \ref{prop_computation_Z_(UxX)}] We recall that, with reference to the diagram (\ref{diag_PxP_G}), we have $E_{\Delta}=b^*(H_1+H_2)-e^*g$ in ${\rm Pic}(\widetilde{(\mathbb P^{n+r})^2_{\Delta_{\mathbb P^{n+r}}}})$.\\
\indent Pulling back the diagram on $U\times X$ and completing the picture (with abuse of notations), we get $$\xymatrix{\mathbb P(e^*\mathcal E_2)\ar[d]_p\ar[r]^q &\mathbb P^{n+r}\\
\widetilde{U\times X}\ar[d]_b\ar[dr]^e &\\
U\times X\ar@{-->}[r] &G(2,n+r+1)}$$ with the exceptional divisor $E_U\simeq \mathbb P(\Omega_{X|U})$ satisfying $E_U=b^*(h_U+h_X)-e^*g$ in ${\rm Pic}(\widetilde{U\times X})$, where $h_U=pr_1^*\mathcal O_U(1)$ and $h_X=pr_2^*\mathcal O_X(1)$.\\
\indent The projection $p$ admits $2$ sections, namely $\sigma_U:(u,x)\mapsto ([\ell_{u,x}],u)$, which comes from the surjection $e^*\mathcal E_2\twoheadrightarrow b^*\mathcal O_U(1)$, and $\sigma_X:(u,x)\mapsto ([\ell_{u,x}],x)$, which comes from the surjection $e^*\mathcal E_2\twoheadrightarrow b^*\mathcal O_X(1)$. Let us denote $D_U=\sigma_U(\widetilde{U\times X})$ and $D_X=\sigma_X(\widetilde{U\times X})$ their respective images. We have $D_U\in |p^*(b^*\mathcal O_X(-1)\otimes \mathcal O_{\widetilde{U\times X}}(E_U))\otimes q^*\mathcal O_{\mathbb P^{n+r}}(1)|$ and $D_X\in |p^*(b^*\mathcal O_U(-1)\otimes \mathcal O_{\widetilde{U\times X}}(E_U))\otimes q^*\mathcal O_{\mathbb P^{n+r}}(1)|$. As $c_1(e^*\mathcal E_2)=e^*g=b^*(h_U+h_X)-E_U$, the following sequence is exact 
\begin{equation}\label{ex_seq_E_2_UxX} 0\rightarrow b^*\mathcal O_X(1)\otimes \mathcal O_{\widetilde{U\times X}}(-E_U)\rightarrow e^*\mathcal E_2\rightarrow b^*\mathcal O_U(1)\rightarrow 0.
\end{equation}

To establish a relation to $Z$, we introduce $$S_U:=\overline{\{(u,x)\in U\times X\backslash \Delta_U,\ P_{d|\ell_{u,x}}\in L_u^{n-r+1}L_xH^0(\mathcal O_{\ell_{u,x}}(d-(n-r+2))\}}\overset{i_S}{\subset}\widetilde{U\times X}.$$

Denoting $$\begin{aligned} \mathcal F' &:=p_*(q^*\mathcal O_{\mathbb P^{n+r}}(d)\otimes \mathcal O_{\mathbb P(e^*\mathcal E_2)}(-D_U-D_X))\\
&\simeq {\rm Sym}^{d-2}e^*\mathcal E_2\otimes b^*(\mathcal O_U(1)\boxtimes \mathcal O_X(1))\otimes \mathcal O_{\widetilde{U\times X}}(-2E_U)
\end{aligned}$$
and $$\begin{aligned}\mathcal F''&:=p_*(q^*\mathcal O_{\mathbb P^{n+r}}(d)\otimes \mathcal O_{\mathbb P(e^*\mathcal E_2)}(-(n-r+1)D_U-D_X))\\
&\simeq {\rm Sym}^{d-(n-r+2)}e^*\mathcal E_2\otimes b^*(\mathcal O_U(1)\boxtimes \mathcal O_X(n-r+1))\otimes \mathcal O_{\widetilde{U\times X}}(-(n-r+2)E_U)
\end{aligned}$$
we have an exact sequence 
\begin{equation}\label{ex_seq_def_F}0\rightarrow \mathcal F''\overset{p_*(\sigma_U^{\otimes (n-r+1)})}{\rightarrow} \mathcal F'\rightarrow \mathcal F\rightarrow 0
\end{equation} 
$\mathcal F$ being a rank $n-r$ vector bundle on $\widetilde{U\times X}$.\\
\indent ``The'' equation $P_d$ of $X$ gives rise to a section of $\mathcal F$ whose zero locus is, by construction $S_U$. So for $X\in |\mathcal O_Y(d)|$ general, $S_U$ is smooth, of dimension $w+n-1-(n-r)=w+r-1$ and $[S_U]=c_{n-r}(\mathcal F)$ in ${\rm CH}_{w+r-1}(\widetilde{U\times X})$.\\

\indent The sections $\sigma_U$ and $\sigma_X$ restrict to sections $\sigma_{U_S}$ and $\sigma_{X_S}$ of $p_S:\mathbb P(i_S^*e^*\mathcal E_2)\rightarrow S_U$. Their images are respectively denoted $D_{U_S}\in |p_S^*i_S^*(b^*\mathcal O_X(-1)\otimes\mathcal O_{\widetilde{U\times X}}(E_U))\otimes q_S^*\mathcal O_{\mathbb P^{n+r}}(1)|$ and $D_{X_S}\in |p_S^*i_S^*(b^*\mathcal O_U(-1)\otimes\mathcal O_{\widetilde{U\times X}}(E_U))\otimes q_S^*\mathcal O_{\mathbb P^{n+r}}(1)|$.\\

\indent Now, denoting $F_1(Y)\subset G(2,n+r+1)$ the variety of lines of $Y$, let us compute the dimension of $e^{-1}(F_1(Y))\cap S_U$ when $U$ is in general position, which we can assume by the moving Lemma.\\
\indent As $X\subset Y$, $e^{-1}(F_1(Y))\cap S_U$ is the zero locus of a regular (for $X\in|\mathcal O_Y(d)|$ general and $U$ in general position) section of $p_{S,*}(q_S^*\oplus_{i=1}^r\mathcal O_{\mathbb P^{n+r}}(d_i)\otimes \mathcal O_{\mathbb P(e^*\mathcal E_2)}(-D_U-D_X))\simeq i_S^*(\oplus_{i=1}^r{\rm Sym}^{d_i-2}e^*\mathcal E_2\otimes b^*(\mathcal O_U(1)\boxtimes \mathcal O_X(1))\otimes \mathcal O_{\widetilde{U\times X}}(-2E_U)$. In particular, $$\begin{aligned}{\rm dim}(e^{-1}(F_1(Y))\cap S_U)&=w+r-1-\sum_{i=1}^r(d_i-1)\\
&=w-1-\sum_{i=1}^r(d_i-2) <w-1\end{aligned}$$ as we have assumed $d_r>2$.\\
\indent So the general line parametrized by $e(S_U)$ is not contained in $Y$ and ${\rm dim}(q_{S_U\cap F_1}^{-1}(Y))<w={\rm dim}(U)$ where the map is defined by $S_U\cap e^{-1}(F_1(Y))\overset{p_{S_U\cap F_1}}{\leftarrow}\mathbb P(i_S^*e^*\mathcal E_2)\overset{q_{S_U\cap F_1}}{\rightarrow} \mathbb P^{n+r}$. In particular, the lines contained in $Y$ do not form a component of $q_S^{-1}(Y)$.\\

\indent As a result, we have $q_S^{-1}(Y)=D_{U_S}\cup D_{X_S}\cup R$ where $R$ is a $w$-dimensional subscheme such that $q_{S,*}(R)=Z_*(U\times Y)$ and $D_{U_S}$ and $D_{X_S}$ have excess dimension $r-1$ and multiplicity $1$.\\
\indent Consider the fiber square $$\xymatrix{q_S^{-1}(Y)\ar[d]_j\ar[r]^k &\mathbb P(i_S^*e^*\mathcal E_2)\ar[d]^{q_S}\\
Y\ar[r]_{i_1} &\mathbb P^{n+r}.}$$
According to \cite[Theorem 6.2]{fulton}, $j_* i_1^![\mathbb P(i_S^*e^*\mathcal E_2)]= i_1^!q_{S,*}[\mathbb P(i_S^*e^*\mathcal E_2)]$ in ${\rm CH}_w(Y)$. The description of the Chow groups of $\mathbb P^{n+r}$ gives $i_1^!q_{S,*}[\mathbb P(i_S^*e^*\mathcal E_2)]\in \mathbb Q\cdot h_Y^{n-w}$.\\
\indent Moreover, (omitting the push-forwards)$i_1^![\mathbb P(i_S^*e^*\mathcal E_2)]=R+[D_{U_S}]\cdot c_{r-1}(Exc(D_{U_S}\backslash R))+[D_{X_S}]\cdot c_{r-1}(Exc(D_{X_S}\backslash R))$ in ${\rm CH}_w(q_S^{-1}(Y))$ where $Exc(D_{U_S}\backslash R)$ and $Exc(D_{X_S}\backslash R)$ denote the respective excess normal sheaves.\\

\begin{lemme}\label{lem_projection_D_U_S} We have $j_*([D_{U_S}]\cdot c_{r-1}(Exc(D_{U_S}\backslash R)))=((n-r)!deg(X)-(n-r+1)!)[U]$ in ${\rm CH}_w(Y)$.
\end{lemme}
\begin{proof} Under the identification $S_U\overset{\sigma_{U_S}}{\simeq} D_{U_S}$, the morphism $j_U:=j_{|D_{U_S}}:D_{U_S}\rightarrow Y$ coincides with the composition $S_U\overset{i_S}{\hookrightarrow}\widetilde{U\times X}\overset{b}{\rightarrow}U\times X\rightarrow U\hookrightarrow Y$. In particular, the support of $j_*([D_{U_S}]\cdot c_{r-1}(Exc(D_{U_S}\backslash R)))$ is contained in $U$ and as that cycle had dimension $w=dim(U)$, it is of the form $\alpha[U]$ for some $\alpha$; so we only have to compute $\alpha$. We do so by computing the $h_U$-degree of the cycle i.e. $\int_{\widetilde{U\times X}}[D_{U_S}]\cdot c_{r-1}(Exc(D_{U_S}\backslash R))b^*h_U^w$.\\
\indent Looking at the commutative diagram $$\xymatrix{D_{U_S}\ar[r]^{i_2}\ar[d]_{j_U} &\mathbb P(i_S^*e^*\mathcal E_2)\ar[d]^{q_S}\\
Y\ar[r]_{i_1} &\mathbb P^{n+r}.}$$ we see that $Exc(D_{U_S}\backslash R))=j_U^*N_{Y/\mathbb P^{n+r}}/N_{i_2}$.\\
\indent We have $N_{i_2}\simeq \sigma_{U_S}^*(p_S^*i_S^*(b^*\mathcal O_X(-1)\otimes \mathcal O_{\widetilde{U\times X}}(E_U))\otimes q_S^*\mathcal O_{\mathbb P^{n+r}}(1))\simeq i_S^*(b^*\mathcal O_X(-1)\otimes \mathcal O_{\widetilde{U\times X}}(E_U))\otimes b^*\mathcal O_U(1))$ since $D_U$ is defined by $e^*\mathcal E_2\twoheadrightarrow b^*\mathcal O_U(1)$. So that, writing $\Pi_{i=1}^r(1+d_ih_U)=\sum_{i=0}^rc_ih_U^i$, $c(Exc(D_{U_S}))=i_S^*\frac{\Pi_{i=1}^r(1+d_ih_U)}{1-(h_X-h_U-E_U)}=i_S^*(\sum_{m\geq 0}\sum_{i=0}^rc_i(h_X-h_U-E_U)^m$ so that $c_{r-1}(Exc(D_{U_S}\backslash R))=i_S^*\sum_{i=0}^rc_ih_U^i(h_X-h_U-E_U)^{r-1-i}$.\\

\indent Now, we recall that $S_U=c_{n-r}(\mathcal F)$ in ${\rm CH}^{n-r}(\widetilde{U\times X})$ (\ref{ex_seq_def_F}). Using (\ref{ex_seq_E_2_UxX}), in $K_0(\widetilde{U\times X})$ we have $e^*\mathcal E_2=b^*\mathcal O_U(1)\oplus b^*\mathcal O_X(1)\otimes \mathcal O_{\widetilde{U\times X}}(-E_U)$. So that 
$$\begin{aligned}\mathcal F'&={\rm Sym}^{d-2}e^*\mathcal E_2\otimes b^*(\mathcal O_U(1)\boxtimes\mathcal O_X(1))\otimes \mathcal O_{\widetilde{U\times X}}(-2E_U)\\
&= \oplus_{i=0}^{d-2}b^*(\mathcal O_U(i+1)\boxtimes \mathcal O_X(d-(i+1)))\otimes \mathcal O_{\widetilde{U\times X}}(-(d-i)E_U)\\
&=\oplus_{i=1}^{d-1}b^*(\mathcal O_U(i)\boxtimes \mathcal O_X(d-i))\otimes \mathcal O_{\widetilde{U\times X}}(-(d+1-i)E_U)
\end{aligned}$$ and 
$$\begin{aligned}\mathcal F''&={\rm Sym}^{d-(n-r+2)}e^*\mathcal E_2\otimes b^*(\mathcal O_U(1)\boxtimes \mathcal O_X(n-r+1))\otimes \mathcal O_{\widetilde{U\times X}}(-(n-r+2)E_U)\\
&=\oplus_{i=0}^{d-(n-r+2)}b^*(\mathcal O_U(i+1)\boxtimes \mathcal O_X(d-(i+1)))\otimes \mathcal O_{\widetilde{U\times X}}(-(d-i)E_U)\\
&=\oplus_{i=1}^{d-(n-r+1)}b^*(\mathcal O_U(i)\boxtimes \mathcal O_X(d-i))\otimes \mathcal O_{\widetilde{U\times X}}(-(d+1-i)E_U).
\end{aligned}$$
As a result, \begin{equation}\label{equal_F_in_K_0}\begin{aligned}\mathcal F=\oplus_{i=d-(n-r)}^{d-1}b^*(\mathcal O_U(i)\boxtimes \mathcal O_X(d-i))\otimes \mathcal O_{\widetilde{U\times X}}(-(d+1-i)E_U)\\
{\rm and\ thus}\ c_{n-r}(\mathcal F)=\Pi_{i=d-(n-r)}^{d-1}(ih_U+(d-i)h_X - (d+1-i)E_U).
\end{aligned}
\end{equation}\\

\indent Now, let us write the codimension $n-1$ cycle $$\begin{aligned}
S_U\cdot c_{r-1}(Exc(D_{U_S}\backslash R))&=\sum_{k+\ell+m=n-1}b_{k,\ell,m}h_U^kh_X^\ell(-E_U)^m\\
&=\sum_{k+\ell=n-1}b_{k,\ell, 0}h_U^kh_X^\ell - j_{E_U,*}(\sum_{m=1}^{n-1}\sum_{k+\ell=n-1-m}b_{k,\ell,m}(-E_{U|E_U})^{m-1} \underset{=j_{E_U}^*h_U^{k+\ell}}{j_{E_U}^*(h_U^kh_X^\ell)}).
\end{aligned}$$
As $h_U^a=0$ for $a>{\rm dim}(U)$ and $b_{|E_U,*}((-E_{U|E_U})^a\cdot b_{|E_U}^*\beta)=0$ for $a< n-2$ (because $E_U\rightarrow U$ is a $\mathbb P^{n-2}$-bundle), we have 
$$\begin{aligned}\int_{\widetilde{U\times X}}[D_{U_S}]\cdot c_{r-1}(Exc(D_{U_S}\backslash R))\cdot h_U^w &= b_{0,n-1,0}\int_Uh_U^w\int_Xh_X^{n-1} -\int_{E_U}(\sum_{m=1}^{n-1}b_{0,0,m}(-E_{U|E_U})^{m-1}j_E^*h_U^w)\\
&=b_{0,n-1,0}deg(U)deg(X) -b_{0,0,n-1}deg(U)
\end{aligned}$$
i.e. $j_{U,*}[D_{U_S}]\cdot c_{r-1}(Exc(D_{U_S}\backslash R))=(b_{0,n-1,0}deg(X)-b_{0,0,n-1})[U]$ in ${\rm CH}_w(Y)$.\\
\indent We are interested, for $b_{0,n-1,0}$, to the maximal $h_X$-power in both $c_{n-r}(\mathcal F)$ and $c_{r-1}(Exc(D_{U_S}\backslash R))$, so $b_{0,n-1,0}=\Pi_{i=d-(n-r)}^{d-1}(d-i)\times c_0 = (n-r)!$.\\
\indent For $b_{0,0,n-1}$, we are interested in the maximal $(-E_U)$-power ($(-E_U)^{n-1}=-(-E_{U|E_U})^{n-2}$) in both $c_{n-r}(\mathcal F)$ and $c_{r-1}(Exc(D_{U_S}\backslash R))$ so $b_{0,0,n-1}=\Pi_{i=d-(n-r}^{d-1}(d+1-i)\times c_0 = (n-r+1)!$.
\end{proof}

Here is the last Lemma to complete the proof of the Proposition.
\begin{lemme}\label{lem_projection_D_X_S}We have $j_*([D_{X_S}]\cdot c_r(Exc(D_{X_S}\backslash R)))=(n-r+1)![U]\ {\rm mod}\ \mathbb Q\cdot h_Y^{n-w}$ in ${\rm CH}_w(Y)$.
\end{lemme}
\begin{proof} Under the identification $D_{X_S}\overset{\sigma_{X_S}}{\simeq} S_U$, the morphism $j_X:=j_{|D_{X_S}}:D_{X_S}\rightarrow Y$ coincides with the composition $S_U\overset{i_S}{\hookrightarrow}\widetilde{U\times X}\overset{b}{\rightarrow}U\times X\rightarrow X\hookrightarrow Y$. In particular the support of $j_*([D_{X_S}]\cdot c_r(Exc(D_{X_S}\backslash R)))$ is contained in $X$.\\
\indent Looking at the commutative diagram
$$\xymatrix{D_{X_S}\ar[r]^{i_2}\ar[d]_{j_X} &\mathbb P(i_S^*e^*\mathcal E_2)\ar[d]^{q_S}\\
Y\ar[r]_{i_1} &\mathbb P^{n+r}.}$$ we see that $Exc(D_{X_S}\backslash R))=j_X^*N_{Y/\mathbb P^{n+r}}/N_{i_2}$.\\
\indent We have $N_{i_2}\simeq \sigma_{X_S}^*(p_S^*i_S^*(b^*\mathcal O_U(-1)\otimes \mathcal O_{\widetilde{U\times X}}(E_U))\otimes q_S^*\mathcal O_{\mathbb P^{n+r}}(1))\simeq i_S^*(b^*\mathcal O_U(-1)\otimes \mathcal O_{\widetilde{U\times X}}(E_U)\otimes b^*\mathcal O_X(1))$. So that writing $\Pi_{i=1}^r(1+d_ih_X)=\sum_{i=0}^rc_ih_X^i$, we get $c(Exc(D_{X_S}\backslash R))=i_S^*\frac{\Pi_{i=1}^r (1+d_ih_X)}{1-(h_U-h_X-E_U)}=i_S^*\sum_{m\geq 0}\sum_{i=0}^rc_ih_X^i(h_U-h_X-E_U)^m$ and thus $$c_{r-1}(Exc(D_{X_S}\backslash R))=i_S^*\sum_{i=0}^rc_ih_X^i(h_U-h_X-E_U)^{r-1-i}.$$

\indent We recall that $[S_U]=c_{n-r}(\mathcal F)$ in ${\rm CH}^{n-r}(\widetilde{U\times X})$ and it has been calculated in (\ref{equal_F_in_K_0}). We can write $$\begin{aligned}D_{X_S}\cdot c_{r-1}(Exc(D_{X_S}\backslash R))&=\sum_{k+\ell+m=n-1}b_{k,\ell,m}h_U^kh_X^\ell(-E_U)^m\\
&=\sum_{k+\ell=n-1}b_{k,\ell,0}h_U^kh_X^\ell -j_{E_U,*}(\sum_{m=1}^{n-1}\sum_{k+\ell=n-1-m}b_{k,\ell,m}(-E_{U|E_U})^{m-1}\underset{=j_{E_U}^*h_U^{k+\ell}}{j_{E_U}^*(h_U^kh_X^\ell)}).\end{aligned}$$
When projected to the second component of $U\times X$, the only part of $\sum_{k+\ell=n-1}b_{k,\ell,0}h_U^kh_X^\ell$ that survives is $b_{w,n-1-w,0}h_U^wh_X^{n-1-w}$ which goes into $\mathbb Q\cdot h_Y^{n-w}$ in ${\rm CH}_w(Y)$.\\
\indent As $E_U\rightarrow U$ is a $\mathbb P^{n-2}$-bundle, $$\begin{aligned}j_{X,*}j_{E_U,*}\sum_{m=1}^{n-1}b_{k,\ell,m}(-E_{U|E_U})^{m-1}j_{E_U}^*b^*h_U^{k+\ell}&=b_{|E_U,*}(\sum_{m=1}^{n-1}b_{k,\ell,m}(-E_{U|E_U})^{m-1}j_{E_U}^*b^*h_U^{k+\ell})\\
&=b_{0,0,n-1}(-E_{U|E_U})^{n-2}.\end{aligned}$$
The coefficient $b_{0,0,n-1}$ is the product of maximal $(-E)$-powers in $c_{n-r}(\mathcal F)$ and $c_{r-1}(Exc(D_{X_S}\backslash R))$. So $b_{0,0,n-1}=(n-r+1)!$.
\end{proof}
\end{proof}

\section*{Funding}
This work was partially supported by Simons Investigators Award HMS, HSE University Basic Research Program, Ministry of Education and Science of Bulgaria, Contract No. DO1-98/13.04.2021 and through the Scientific Program “Enhancing the Research Capacity in Mathematical Sciences (PIKOM)” No. DO1-67/05.05.2022.
\section*{Acknowledgments}
I would like to thank Emma Brakkee for suggestions on a preliminary version of this work. I would like to thank also Claire Voisin for asking me this question some years ago, and Lie Fu for useful discussions around a previous attempt to solve it.\\

\textit{}\\
\noindent \begin{tabular}[t]{l}
\textit{rene.mboro@polytechnique.edu}\\
UMiami Miami, HSE Moscow,\\
Institute of Mathematics and Informatics, Bulgarian Academy of Sciences,\\
Acad. G. Bonchev Str. bl. 8, 1113, Sofia, Bulgaria.
\end{tabular}
\end{document}